\documentclass[11pt,a4paper]{amsart}
\usepackage{latexsym}
\usepackage{prettyref}
\usepackage{float}
\usepackage{amsmath}
\usepackage[dvips]{graphicx}
\usepackage{graphics}
\usepackage{amssymb}
\usepackage{graphicx,psfrag}

\usepackage{hyperref}

\begin{document}

\newcommand{\RR}{\mathbb{R}^{2}}
\newcommand{\hh}[1]{\mathcal{H}_{\delta}^{#1}}
\newcommand{\HH}[1]{\mathcal{H}^{#1}}
\newcommand{\h}{\mathfrak{h}}
\newcommand{\g}{\mathfrak{g}}
\newcommand{\e}{\varepsilon}
\newcommand{\N}{\mathbb{N}}
\newcommand{\R}{\mathbb{R}}
\newcommand{\RN}{\mathbb{R}^{n}}
\newcommand{\s}{\mathbb{S}}
\newcommand{\sub}{\subseteq}
\newcommand{\diam}{\text{diam}}
\newcommand{\D}{\displaystyle}
\newcommand{\Z}{\mathbb{Z}}
\newcommand{\Q}{\mathbb{Q}}
\newcommand{\gap}{\underline{\Delta}}
\newcommand{\GAP}{\overline{\Delta}}
\newcommand{\E}[1]{\mathcal{E}(#1)}
\newcommand{\DN}[1]{\textbf{DN}(#1)}
\renewcommand{\a}{\mathfrak{a}}
\renewcommand{\b}{\mathfrak{b}}
\newcommand{\f}{\mathfrak{f}}
\newcommand{\id}{\textbf{id}}
\renewcommand{\i}{\textbf{id}}
\renewcommand{\H}{\mathbb{H}}
\newcommand{\J}{\mathbb{J}}
\newcommand{\NN}{\mathcal{N}_\delta}
\newcommand{\lbd}{\underline{\dim}_B}
\newcommand{\ubd}{\overline{\dim}_B}
\newcommand{\bd}{\dim_B}

\date{}

\newtheorem{theorem}{Theorem}[section]
\newtheorem{lemma}[theorem]{Lemma}
\newtheorem{corollary}[theorem]{Corollary}
\newtheorem{proposition}[theorem]{Proposition}
\newtheorem{claim}[theorem]{Claim}

\newtheorem{conjecture}[theorem]{Conjecture}

\theoremstyle{definition}
\newtheorem{definition}[theorem]{Definition}
\newtheorem{example}[theorem]{Example}
\newtheorem{remark}[theorem]{Remark}

\title{Furstenberg sets for a fractal set of directions}

\author{Ursula Molter and Ezequiel Rela}

\address{Departamento de
Matem\'atica \\ Facultad de Ciencias Exactas y Naturales\\ Universidad
de Buenos Aires\\ Ciudad Universitaria, Pabell\'on I\\ 1428 Capital
Federal\\ ARGENTINA\\ and CONICET, Argentina}
 \email[Ursula Molter]{umolter@dm.uba.ar}
\email[Ezequiel Rela]{erela@dm.uba.ar}

\thanks{This research  is partially supported by
Grants: PICT2006-00177, PIP 11220080100398
 and UBACyT X149}

\maketitle

\begin{abstract}
In this note we study the behavior of the size of Furstenberg sets with respect to the size of the set of directions defining it. For any pair $\alpha,\beta\in(0,1]$, we will say that a set $E\subset \R^2$ is an $F_{\alpha\beta}$-set if there is a subset $L$ of the unit circle of Hausdorff dimension at least $\beta$ and, for each direction $e$ in $L$, there is a line segment $\ell_e$ in the direction of $e$ such that the Hausdorff dimension of the set $E\cap\ell_e$ is equal or greater than $\alpha$. The problem is considered in the wider scenario of generalized Hausdorff measures, giving estimates on the appropriate dimension functions for each class of Furstenberg sets. As a corollary of our main results, we obtain that $\dim(E)\ge\max\left\{\alpha+\frac{\beta}{2} ; 2\alpha+\beta -1\right\}$ for any $E\in F_{\alpha\beta}$. In particular we are able to extend previously known results to the ``endpoint'' $\alpha=0$ case.
\end{abstract}

\section{Introduction}

In this article we are interested in the study of dimension properties of Furstenberg sets associated to fractal sets of directions. Let us introduce the definition of our object of study. In the sequel, we will denote with $\dim(E)$ the Hausdorff dimension of the set $E$.
\begin{definition}\label{def:alphabeta}
For  $\alpha,\beta$ in $(0,1]$, a subset $E$ of $\RR$ will be called an $F_{\alpha\beta}$-set if there is a subset $L$ of the unit circle such that $\dim(L)\ge\beta$ and, for each direction $e$ in $L$, there is a line segment $\ell_e$ in the direction of $e$ such that the Hausdorff dimension of the set $E\cap\ell_e$ is equal or greater than $\alpha$. 
\end{definition}
This generalizes the classical definition of Furstenberg sets, when the whole circle is considered as set of directions. For $L=\s$, which is a particular case of $\beta=1$, we recover the classical class $F_\alpha$ of $\alpha$-Furstenberg sets, and the best known result is 
\begin{equation}\label{eq:dim}
\max\left\{\alpha+\frac{1}{2} ; 2\alpha\right\}\le \gamma(\alpha)\le\frac{1}{2}+\frac{3}{2}\alpha,\qquad \alpha\in(0,1].
\end{equation} 
where $  \gamma(\alpha)=\inf \{\dim(E): E\in F_\alpha\}$. In \cite{mr10a} and \cite{mr10b} the above inequalities are proved in the general setting of dimension functions, allowing the extension to the endpoint $\alpha=0$ for some class of generalized Furstenberg sets.

Unavoidable references on this matter are \cite{wol99a}, \cite{wol03}, \cite{KT01} and \cite{taofinite}. 

The purpose of this note is to study how the parameter $\beta$ affects the bounds above. Moreover, by using general Hausdorff measures, we will extend the inequalities \prettyref{eq:dim} to the zero dimensional case.

From our results we will derive the following proposition.
\begin{proposition}\label{pro:fab}
For any set $E\in F_{\alpha\beta}$, we have that 
\begin{equation}\label{eq:dim_ab}
\dim(E)\ge \max\left\{\alpha+\frac{\beta}{2} ; 2\alpha+\beta -1\right\},\qquad \alpha,\beta>0.
\end{equation} 
\end{proposition}
It is not hard to prove \prettyref{pro:fab} directly, but we will study this problem in a wider scenario and derive it as a corollary. We also remark that our results are consistent with the ones in \cite{mit02}, where the author proves, essentially, the second bound for the case $\alpha=1$, $\beta\in(0,1]$.

There is a natural way to generalize this problem by looking at dimension functions that are not necessarily power functions (\cite{Hau18}). Let us begin with the notion of dimension functions.

\subsection{Dimension Functions}
\begin{definition}
The following class of functions will be called \textit{dimension functions}. 
\begin{equation*}
 \mathbb{H}:=\{h:[0,\infty)\to[0:\infty), 
 \text{non-decreasing, right continuous, } 
h(0)=0 \}.
\end{equation*} 
\end{definition}
The important subclass of those $h\in\mathbb{H}$ that satisfy a doubling condition will be denoted by $\mathbb{H}_d$:
\begin{equation*}
	\mathbb{H}_d:=\left\{h\in\mathbb{H}: h(2x)\le C h(x) \text{ for some }C>0\right\}.
\end{equation*} 
\begin{remark}
Clearly, if $h\in\mathbb{H}_d$, the same inequality will hold (with some other constant) if 2 is replaced by any other $\lambda>1$. We also remark that any concave function trivially belongs to $\mathbb{H}_d$. Also note that the monotonicity of $h$ implies that $C\ge 1$.
\end{remark}

If one only looks at the power functions, there is a natural total order given by the exponents. If we denote with $h_\alpha(x)=x^\alpha$, then $h_\alpha$ is, in some sense, {\em smaller} than $ h_\beta$ if and only if $\alpha<\beta$. In $\mathbb{H}$ we also have a natural notion of order, but we can only obtain a \textit{partial} order.
\begin{definition}\label{def:order}
Let $g,h$ be two dimension functions. We will say that $g$ is dimensionally smaller than $h$ and write $g\prec h $ if and only if
	\begin{equation*}
	\lim_{x\to 0^+}\dfrac{h(x)}{g(x)}=0.
	\end{equation*}	
\end{definition}

We also remark that we will be particularly interested in the special subclass of dimension functions that allows us to classify zero dimensional sets, that means, that $h$ is in this class if it is smaller than any of the functions $x^\alpha$, $\alpha>0$.

\begin{definition}
A function $h\in\H$ will be called ``zero dimensional dimension function'' if $h\prec h_\alpha$ for any $\alpha>0$. 
We will denote by $\mathbb{H}_0$ the subclass of those functions.
\end{definition}

As usual, the $h$-dimensional (outer) Hausdorff  measure $\HH{h}$ will be defined as follows. For a set $E\subseteq\R^n$ and $\delta>0$, write
\begin{equation*}
	\hh{h}(E)=\inf\left\{\sum_i h(\diam(E_i)):E\subset\bigcup_i^\infty E_i,  \diam(E_i)<\delta \right\}.
\end{equation*}
The $h$-dimensional Hausdorff measure $\HH{h}$ of $E$ is defined by
\begin{equation*}
	\HH{h}(E)=\sup_{\delta>0 }\hh{h}(E).
\end{equation*} 
We remark that, even though they would not lead to the exact same measures, we will consider functions $g,h$ such that there exist constants $c,C$ with $0<c\le \frac{g(x)}{h(x)}\le C<\infty$ for all $x>0$ to be equivalent. In that case we write $g\equiv h$.

To measure the ``distance'' between to dimension functions, we introduce the following notion:
\begin{definition}
 Let $g,h\in\H$ with $g\prec h$. Define the ``gap'' between $g$ and $h$ as
\begin{equation}
\Delta(x)=\frac{h(x)}{g(x)}.
\end{equation} 
\end{definition}
From this definition and the definition of partial order, we always have that $\lim_{x\to 0}\Delta(x)=0$, and therefore the speed of convergence to zero can be seen as a notion of distance between $g$ and $h$. 

Now we present the problem. Let us begin with the definition of $F_{\h\g}$-sets. Let $\h$ and $\g$ be two dimension functions. A set $E\subseteq\RR$ is a Furstenberg set of type $\h\g$, or an $F_{\h\g}$-set, if there is a subset $L$ of the unit circle such that $\HH{\g}(L)>0$ and, for each direction $e$ in $L$,  there is a line segment $\ell_e$ in the direction of $e$ such that  $\HH{\h}(\ell_e \cap E)>0$. 

Note that this hypothesis is stronger than the one used to define the original Furstenberg-$\alpha$ sets. However, the hypothesis $\dim(E\cap\ell_e)\ge \alpha$ is equivalent to $\HH{\beta}(E\cap\ell_e)>0$ for any $\beta$ smaller than $\alpha$. If we use the wider class of dimension functions introduced above, the natural way to define $F_\h$-sets would be to replace the parameters $\beta<\alpha$ with two dimension functions satisfying the relation $h\prec \h$. But requiring  $E\cap\ell_e$ to have positive $\HH{h}$ measure for any $h\prec \h$ implies that it has also positive $\HH{\h}$ measure (Theorem 42, \cite{rog70}). Therefore, this definition is the natural generalization of the $F^+_{\alpha\beta}$ class defined below.

\begin{definition}\label{def:alphabeta+}
For each pair $\alpha,\beta$ in $(0,1]$, a subset $E$ of $\RR$ will be called an $F^+_{\alpha\beta}$-set if there is a subset $L$ of the unit circle such that $\HH{\beta}(L)>0$ and, for each direction $e$ in $L$, there is a line segment $\ell_e$ in the direction of $e$ such that $\HH{\alpha}(\ell_e\cap E)>0$.
\end{definition}

Now, for the sake of clarity in the proof of our results, we will perform the same reduction made in \cite{mr10a}. A standard pigeonhole argument allows us to work with the following definition.

\begin{definition}\label{def:Fhg}
Let $\h$ and $\g$ be two dimension functions. A set $E\subseteq\RR$ is a Furstenberg set of type $\h\g$, or an $F_{\h\g}$-set, if there is a subset $L$ of the unit circle such that  $\HH{\g}(L)>0$ and, for each direction $e$ in $L$, there is a line segment $\ell_e$ in the direction of $e$ such that  $\hh{\h}(\ell_e \cap E)>1$ for all $\delta<\delta_E$ for some $\delta_E>0$ with $\delta_E$ depending only on $E$.
\end{definition}

Following the intuition suggested by \prettyref{pro:fab}, one could conjecture that if $E$ belong to the class $F_{\h\g}$ then  an appropriate dimension function for $E$ should be dimensionally greater than $\frac{\h^2\g}{\i}$ and $\h\sqrt{\g}$ (where $\i$ is the identity function). This will indeed be the case, and we will provide some estimates on the gap between those conjectured dimension functions and a generic test function $h\in\H$ to ensure that $\HH{h}(E)>0$. In addition we illustrate with some examples.
We will consider the two results separately. Namely, for a given pair of dimension functions $\g\in\H$ and $\h\in\H_d$, in \prettyref{sec:h2g/id} we obtain sufficient conditions on a test dimension function $h\in\H$, $h\succ \frac{\h^2\g}{\i}$ to ensure that $\HH{h}(E)>0$ for any set $E\in F_{\h\g}$. In \prettyref{sec:hsqrtg} we consider the analogous problem for $h\succ \h\sqrt{\g}$. The next section summarizes some preliminary results to be used in our proofs and additional notation. Finally, in \prettyref{sec:size} we briefly discuss the appropriate notion of size for the set of directions defining the Furstenberg classes.

\section{Preliminaries}\label{sec:prelim}

In this section we include some preliminary and technical results needed in the sequel. We will use the notation $A\lesssim B$ to indicate that there is a constant $C>0$ such that $A\le C B$, where the constant is independent of $A$ and $B$. By $A\sim B$ we mean that both $A\lesssim B$ and $B\lesssim A$ hold. As usual, by a $\delta$-covering of a set $E$ we mean a covering of $E$ by sets $U_i$ with diameters not exceeding $\delta$.

In \prettyref{sec:h2g/id} the main tool will be an $L^2$ estimate for the Kakeya maximal function for general measures. For an integrable function on $\RN$, the Kakeya maximal function at  scale $\delta$ will be $\mathcal{K}_\delta(f):\s^{n-1}\to \R$,
\begin{equation}
	\mathcal{K}_\delta(f)(e)=\sup_{x\in\RN}\frac{1}{|T_e^\delta(x)|}\int_{T_e^\delta(x)}|f(x)|\ dx\qquad e\in\s^{n-1},\nonumber
\end{equation}
where $T_e^\delta(x)$ is a $1\times \delta$-tube (by this we mean a tube of length 1 and cross section of radius $\delta$) centered at $x$ in the direction $e$. 

The estimate we need is the main result of \cite{mit02}. There the author proves (Theorem 3.1) the following.
\begin{proposition}\label{pro:maximalgeneral}
Let $\mu$ be a Borel probability measure on $\s$ such that $\mu(B(x,r))\lesssim \varphi(r)$ for some non-negative function $\varphi$ for all $r\ll 1$. Define the Kakeya maximal operator $\mathcal{K}_\delta$ as usual:
\[
\mathcal{K}_\delta(f)(e)=\sup_{x\in\RN}\frac{1}{|T_e^\delta(x)|}\int_{T_e^\delta(x)}|f(x)|\ dx,\qquad e\in\s^{n-1}.
\]
Then we have the estimate
\begin{equation}\label{eq:maximalgeneral}
 \|\mathcal{K}_\delta\|^2_{L^2(\RR)\to L^2(\s,d\mu)}\lesssim C(\delta)=\int_\delta^1\frac{\varphi(u)}{u^2}du.
\end{equation}
\end{proposition}

\begin{remark}
 It should be noted that if we choose $\varphi(x)=x^s$, then we obtain as a corollary that 
\begin{equation}\label{eq:maximalbetapower}
 \|\mathcal{K}_\delta\|^2_{L^2(\RR)\to L^2(\s,d\mu)}\lesssim \delta^{s-1}.
\end{equation}
In the special case of $s=1$, the bound has the known logarithmic growth:
\[
 \|\mathcal{K}_\delta\|^2_{L^2(\RR)\to L^2(\s,d\mu)}\sim \log(\frac{1}{\delta}).
\]
\end{remark}

This result will be used in \prettyref{sec:h2g/id}, where the hypotheses imposed on a set $E$ for being an $F_{\h\g}$ set guarantee, via Frostman's lemma, that there exists a probability measure $\mu$ on the set of directions $L$ with $\mu(B_r)\lesssim \g(r)$ for any ball $B_r$ (see \cite{mat95}). Let us remark that \prettyref{eq:maximalbetapower} suggests that the constant $C(\delta)$ plays, in the general case, the role of $\frac{\g}{\i}(\delta)$.

In \prettyref{sec:hsqrtg} we perform a more combinatorial kind of proof. We introduce the notion of $\delta$-entropy of a set $E$ in the next definition

\begin{definition}
 Let $E\subset \RN$ and $\delta\in\R_{>0}$. The $\delta$-entropy  of $E$ is the maximal possible cardinality of a $\delta$-separated subset of $E$. We will denote this quantity with $\NN(E)$.
\end{definition}

The main idea is to relate the $\delta$-entropy to some notion of size of the set. Clearly, the entropy is essentially the Box dimension or the Packing dimension of a set (see \cite{mat95} or \cite{fal03} for the definitions) since both concepts are defined in terms of separated $\delta$ balls with centers in the set. However, for our proof we will need to relate the entropy of a set to some quantity that has the property of being (in some sense) stable under countable unions. One choice is therefore the notion of Hausdorff content, which enjoys the needed properties: it is an outer measure, is finite, and reflects the entropy of a set in the following manner. Recall that the $\g$-dimensional Hausdorff content of a set $E$ is defined as 
\begin{equation}\label{eq:Hauscontent}
\HH{\g}_\infty(E)=\inf\left\{\sum_i\g(\diam(U_i): E\subset \bigcup_i U_i\right\}.
\end{equation} 

Note that the $\g$-dimensional Hausdorff content $\HH{\g}_\infty$ is clearly not the same than the $\g$-dimensional Hausdorff measure $\HH{\g}$. In fact, they are the measures obtained by applying Method I and Method II (see \cite{mat95}) respectively to the premeasure that assigns to a set $A$ the value $\g(\diam(A))$.

For future reference, we state the following estimate for the $\delta$-entropy of a set with positive $\g$-dimensional Hausdorff content as a lemma.
\begin{lemma}\label{lem:entropy}
Let $\g\in\H$ and let $A$ be any set. Let  $\NN(A)$ be the $\delta$-entropy of $A$. Then $\NN(A)\ge\frac{\HH{\g}_\infty(A)}{\g(\delta)}$.
\end{lemma}
\begin{proof}
Let $\{x_i\}^N_{i=1}$ be a maximal $\delta$-separated subset. By maximality, we can cover $A$ with balls $B(x_i,\delta)$. Therefore, for the $\g$-dimensional Hausdorff content $\HH{\g}_\infty$, we have the bound
\begin{equation}
\HH{\g}_\infty(A)\le \sum_i^N \HH{\g}_\infty (B(x_i,\delta)) \le  N \g(\delta)                                                         \end{equation} 
and it follows that $\NN(A)\ge N\ge \frac{\HH{\g}_\infty(A)}{\g(\delta)}$. 
\end{proof}
Of course, this result is meaningful when $\HH{\g}_\infty(A)>0$. We will use it in the case $\HH{\g}(A)>0$ which is equivalent to $\HH{\g}_\infty(A)>0$. For a detailed study of the properties of $\HH{\g}$ and $\HH{\g}_\infty$ see \cite{del02} and \cite{del03}.

 Note that the lemma above only requires the finiteness and the subadditivity of the Hausdorff content. The relevant feature that will be needed in our proof is the $\sigma$-subadditivity, which is a property that the Box dimension does not share.

Now we introduce the following notation and a technical lemma.
\begin{definition}
Let $\mathfrak{b}=\{b_k\}_{k\in\N}$ be a decreasing sequence with $\lim b_k=0$. For any family of balls $\mathcal{B}=\{B_j\}$ with $B_j=B(x_j;r_j)$, $r_j\le 1$, and for any set $E$, we define
\begin{equation}\label{eq:jkb}
J^\b_k:=\{j\in\N:b_k< r_j\le b_{k-1}\},
\end{equation}
and 
\begin{equation}\label{eq:ek}
E_k:=E\cap \bigcup_{j\in J^\mathfrak{b}_k}B_j.
\end{equation}
In the particular case of the dyadic scale $\mathfrak{b}=\{2^{-k}\}$, we will omit the superscript and denote
\begin{equation}\label{eq:jk}
J_{k}:=\{j\in\N:2^{-k}<r_{j}\le2^{-k+1}\}.
\end{equation} 

\end{definition}

The next lemma introduces a technique used in \cite{mr10a} to decompose  the set of all directions. 
\begin{lemma}\label{lem:part}
	Let $E$ be an $F_{\h\g}$-set for some $\h,\g\in\mathbb{H}$ with the directions in $L\subset \s$ and let  $\a=\{a_k\}_{k\in\N}\in\ell^1$ be a non-negative sequence. Let $\mathcal{B}=\{B_j\}$ be a $\delta$-covering of $E$ with $\delta< \delta_E$ and let $E_k$ and $J_k$ be as above. Define
\begin{equation}
	L_k:=\left\{e\in\s: \hh{\h}(\ell_e\cap E_k)\ge \frac{a_k}{2\|\a\|_1}\right\}.\nonumber
\end{equation}
Then $L=\cup_k L_k$.
\end{lemma}
The proof follows directly from the summability of $\a$.

\section{The Kakeya type bound }\label{sec:h2g/id}

In this section we prove a generalized version of the announced bound $\dim(E)\ge 2\alpha+\beta-1$ for $E\in F_{\alpha\beta}$. We have the following theorem.
\begin{theorem}[$\h\g\to \frac{\h^2\g}{\id}$]\label{thm:hgtoh2g/x}
	Let $\g\in\H$, $\h\in\mathbb{H}_d$ be two dimension functions and let $E$ be an $F_{\h\g}$-set. For $\delta>0$, let $C(\delta)$ be as in \prettyref{eq:maximalgeneral}. For any $h\in\mathbb{H}$ such that ${\D\sum_k} \sqrt{\frac{\h^2(2^{-k})C(2^{-k})}{h(2^{-k})}}<\infty$, $\HH{h}(E)>0$.
\end{theorem}

\begin{proof}
Let $E\in F_{\h\g}$ and let $\{B_j\}_{j\in\N}$ be a covering of $E$ by balls with $B_j=B(x_j;r_j)$. We need to bound $\sum_j h(2r_j)$ from below. Since $h$ is non-decreasing, it suffices to obtain the  bound 

\begin{equation}\label{eq:sum}
	\sum_j h(r_j)\gtrsim 1
\end{equation}  
for any $h\in \mathbb{H}$ satisfying the hypothesis of the theorem. 

Define $\a=\{a_k\}$ by $a^2_k=\frac{\h^2(2^{-k})C(2^{-k})}{h(2^{-k})}$. Also define, as in the previous section, for each $k\in\N$, $J_k=\{j\in \N: 2^{-k}< r_j\le 2^{-k+1}\}$ and $E_k=E\cap\D\cup_{j\in J_k}B_j$.
Since by hypothesis $\a\in\ell^1$, we can apply \prettyref{lem:part} to obtain the decomposition of the set of directions as $L=\bigcup_k L_k$ associated to this choice of $\a$.

We will apply the maximal function inequality to a weighted union of indicator functions. For each $k$, let $F_k=\D\bigcup_{j\in J_k}B_j$ and define the function 
\[
f:=\h(2^{-k})2^k\D\chi_{F_k}.
\]

We will use the $L^2$ norm estimates for the maximal function. We can compute directly the $L^2$ norm of $f$:
\begin{eqnarray*}
    	\|f\|_{2}^{2} & = & \h^2(2^{-k})2^{2k}\int_{\cup_{J_k}B_{j}}dx\\
                  & \lesssim & \h^2(2^{-k})2^{2k}\sum_{j\in J_{k}}r_j^2 \\
		& \lesssim & \h^2(2^{-k}) \#J_k,
	\end{eqnarray*}
since $r_j\le 2^{-k+1}$ for $j\in J_k$. Therefore
\begin{equation}\label{eq:L2above}
	\|f\|_{2}^{2}\lesssim \#J_k \h^2(2^{-k}).
\end{equation} 
The same arguments used in the proof of Theorem 3.1 in \cite{mr10a} allows us to obtain a lower bound for the maximal function. Essentially, the maximal function is pointwise bounded from below by the average of $f$ over the tube centered on the line segment $\ell_e$ for any $e\in L_k$. Therefore, we have the following bound for the $(L^2,\mu)$ norm. Here, $\mu$ is a measure supported on $L$ that obeys the law $\mu(B(x,r)\le\g(r)$ for any ball $B(x,r)$ given by Frostman's lemma.
\begin{equation}\label{eq:maximalbelow}
	\|\mathcal{K}_{2^{-k+1}}(f)\|_{L^2(d\mu)}^2\gtrsim a_k^2\mu(L_k)=\frac{\mu(L_k)\h^2(2^{-k})C(2^{-k})}{h(2^{-k})}.
\end{equation} 
Combining \prettyref{eq:maximalbelow} with the maximal inequality \prettyref{eq:maximalgeneral}, we obtain
\begin{equation*}
	\frac{\mu(L_k)\h^2(2^{-k})C(2^{-k})}{h(2^{-k})}\lesssim \|\mathcal{K}_{2^{-k+1}}(f)\|_2^2\lesssim C(2^{-k+1})\|f\|_2^2\le C(2^{-k})\|f\|_2^2.
\end{equation*} 
We also have the bound \prettyref{eq:L2above}, which implies that 
\begin{equation}
	\frac{\mu(L_k)}{h(2^{-k})}\lesssim \#J_k.\nonumber
\end{equation}

Now we are able to estimate the sum in \prettyref{eq:sum}. Let $h$ be a dimension function satisfying the hypothesis of \prettyref{thm:hgtoh2g/x}. We have
\begin{eqnarray*}
	\sum_j h(r_j)	& \ge & \sum_k h(2^{-k})\#J_k\\
			& \gtrsim & \sum_k \mu(L_k) \ge\mu(L)>0.
\end{eqnarray*}
\end{proof}

\begin{corollary}\label{cor:coroh2g/x}
	Let $E$ an $F^+_{\alpha\beta}$-set. If $h$ is any dimension function satisfying
\begin{equation}
 h(x)\ge C x^{2\alpha+\beta-1} \log^{\theta}(\frac{1}{x})
\end{equation}  
for $\theta>2$, then $\HH{h}(E)>0$.
\end{corollary}

\begin{proof}
 It follows directly, since in this case we have $C(\delta)\lesssim\delta^{\beta-1}$, and therefore the sum in \prettyref{thm:hgtoh2g/x} is 
\begin{eqnarray*}
\sum_k \sqrt{\frac{\h^2(2^{-k})C(2^{-k})}{h(2^{-k})}} &\lesssim& \sum_k \sqrt{\frac{2^{-k2\alpha}2^{-k(\beta-1)}}{h(2^{-k})}}\\
& \le & \sum_k \sqrt{\frac{2^{-k(2\alpha+\beta-1)}}{(2^{-k})^{2\alpha+\beta-1} \log^{\theta}(2^k)}}\\
& = & \sum_k \frac{1}{k^{\frac{\theta}{2}}}<\infty.
\end{eqnarray*} 
\end{proof}

\begin{remark}
Note that the bound $\dim(E)\ge 2\alpha+\beta-1$ for $E\in F_{\alpha\beta}$ follows directly from this last corollary.
\end{remark}

\section{The combinatorial bound}\label{sec:hsqrtg}

In this section we deal with the bound $\h\g\to\h\sqrt{\g}$, which is the significant bound near the endpoint $\alpha=\beta=0$ and generalizes the bound $\dim(E)\ge \frac{\beta}{2}+\alpha$ for $E\in F_{\alpha\beta}$. Note that the second bound in \prettyref{eq:dim_ab} is meaningless for small values of $\alpha$ and $\beta$. We will consider separately the cases of $\h$ being zero dimensional or positive dimensional. In the next theorem, the additional condition on $\h$ reflects the positivity of the dimension function.

We believe that it would be helpful to cite, without the proofs, two relevant lemmas used in \cite{mr10a}. 

The first is a ``splitting lemma'', which says that a linear set with positive $\h$-dimensional mass can be splitted into two well separated linear subsets.

\begin{lemma}\label{lem:split}
 Let $\h\in\mathbb{H}$, $\delta>0$, $I$ an interval and $E\sub I$. Let $\eta>0$ be such that $\h^{-1}(\frac{\eta}{8})<\delta$ and $\hh{\h}(E)\ge\eta>0$. Then there exist two subintervals $I^-$, $I^+$ that are $\h^{-1}(\frac{\eta}{8})$-separated and  with $\hh{\h}(I^{\pm}\cap E)\gtrsim \eta$.
\end{lemma}

The second lemma is the combinatorial ingredient in the proof of both \prettyref{thm:hgtohsqrtg} and \prettyref{thm:hgtohsqrtg,hzero,gpos}. This lemma provides an estimate on the number of lines with certain separation that intersect two balls of a given size.

\begin{lemma}\label{lem:conobolas}
Let $\mathfrak{b}=\{b_k\}_{k\in\N}$ be a decreasing sequence with $\lim b_k=0$. Given a family of balls $\mathcal{B}=\{B(x_j;r_j)\}$, we define $J^\b_k$ as in \prettyref{eq:jkb} and let $\{e_i\}_{i=1}^{M_k}$ be a $b_k$-separated set of directions. Assume that for each $i$ there are two line segments $I_{e_i}^+$ and $I_{e_i}^-$ lying on a line in the direction $e_i$  that are $s_k$-separated for some given $s_k$ 
Define $\Pi_k=J^\b_{k}\times
J^\b_{k}\times\{1,..,M_k\}$ and $\mathcal{L}^\b_k$ by

\begin{equation}
\mathcal{L}^\b_k:=\left\{(j_{+},j_{-},i)\in \Pi_k:
    	I_{e_{i}}^-\cap B_{j_-}\neq\emptyset\ 
	I_{e_{i}}^+\cap B_{j_+}\neq\emptyset
 \right\}.\nonumber
\end{equation} 
If $\frac{1}{5}s_k>b_{k-1}$ for all $k$, then
\begin{equation}
\#\mathcal{L}^\b_k\lesssim\frac{b_{k-1}}{b_k}\frac{1}{s_k}\left(\#J^\b_{k}\right)^{2}.\nonumber
\end{equation} 
\end{lemma}

With these two lemmas we are now ready to prove the main result of this section. We have the following theorem. Recall that $h_\alpha(x)=x^\alpha$.

\begin{theorem}[$\h\g\to\h\sqrt{\g}$, $\h\succ h_\alpha$]\label{thm:hgtohsqrtg}
	Let $\g\in\H$, $\h\in\mathbb{H}_d$ be two dimension functions such that $\h(x)\lesssim x^\alpha$ for some $0<\alpha<1$ and let $E$ be an $F_{\h\g}$-set. Let $h\in\mathbb{H}$ with $h\prec \h\sqrt{\g}$.  If ${\D\sum_k} \left(\frac{\h(2^{-k})\sqrt{\g}(2^{-k})}{h(2^{-k})}\right)^{\frac{2\alpha}{2\alpha+1}}<\infty$, then $\HH{h}(E)>0$.
\end{theorem}

\begin{proof}
Let $E\in F_{\h\g}$ and let $\{B_j\}_{j\in\N}$ be a covering of $E$ by balls with $B_j=B(x_j;r_j)$. Define $\Delta=\frac{\h\sqrt{\g}}{h}$ and consider the sequence $\a=\left\{\Delta^{\frac{2\alpha}{2\alpha+1}}(2^{-k})\right\}_{k}$. Also define, as in the previous section, for each $k\in\N$, $J_k=\{j\in \N: 2^{-k}< r_j\le 2^{-k+1}\}$ and $E_k=E\cap\D\cup_{j\in J_k}B_j$.
Since by hypothesis $\a\in\ell^1$, we can apply \prettyref{lem:part} to obtain the decomposition of the set of directions as $L=\bigcup_k L_k$ associated to this choice of $\a$, where $L_k$ is defined as 

\begin{equation}
	L_k:=\left\{e\in\s: \hh{\h}(\ell_e\cap E_k)\ge \frac{a_k}{2\|\a\|_1}\right\}.\nonumber
\end{equation}
We can apply \prettyref{lem:split} with $\eta=\frac{a_k}{2\|\a\|_1}$ to $\ell_e\cap E_k$. Therefore we obtain two intervals $I_e^-$ and $I_e^+$, contained in $\ell_e$ with
\begin{equation}
	\hh{\h}(I^{\pm}_e\cap E_k)\gtrsim a_k\nonumber
\end{equation} 
that are $\h^{-1}(ra_k)$-separated for $r=\frac{1}{16\|\a\|_1}$.

Now, let $\{ e^k_{j}\}_{j=1}^{N_k}$ be a $2^{-k}$-separated subset of $L_k$. Taking into account the estimate for the entropy given in \prettyref{lem:entropy}. We obtain then that
\begin{equation}\label{eq:entropy_Hinfty}
 N_k\gtrsim\frac{\HH{\g}_\infty(L_k)}{\g(2^{-k})}.
\end{equation} 
Define $\Pi_k:=J_{k}\times
J_{k}\times\{1,..,N	_k\}$ and 
\begin{equation}\label{eq:tauk}
\mathcal{T}_k:=\left\{(j_{-},j_{+},i)\in \Pi_k:
    	I_{e_{i}}^-\cap E_{k}\cap B_{j_-}\neq\emptyset\ 
	I_{e_{i}}^+\cap E_{k}\cap B_{j_+}\neq\emptyset
 \right\}.
\end{equation} 
The idea is to count the elements of $\mathcal{T}_k$ in two ways. If we fix a pair $j_{-}$ and $j_{+}$ and count for how many values of $i$ the triplet $(j_{-},j_{+},i)$ belongs to $\mathcal{T}_k$, we obtain, by using \prettyref{lem:conobolas} for the choice $\b=\{2^{-k}\}$, that 
\begin{equation}\label{eq:jfijo}
\#\mathcal{T}_k\lesssim\frac{1}{\h^{-1}(ra_k)}\left(\#J_{k}\right)^{2}.
\end{equation} 
Second, fix $i$. In this case, we have by hypothesis that  $\hh{\h}(I_{e_i}^{+}\cap E_{k})\gtrsim a_k$, so $\sum_{j_{+}}\h(r_{j_+})\gtrsim a_k$. Therefore, 

\begin{equation}
 a_k\lesssim\sum_{(j_-,j_+,i)\in\mathcal{T}_k}\h(r_{j_+})\le K \h(2^{-k}),\nonumber
\end{equation} 

where $K$ is the number of elements of the sum. Therefore $K\gtrsim \frac{a_k}{\h(2^{-k})}$. The same holds for $j_{-}$, so
\begin{equation}\label{eq:ifijo}
\#\mathcal{T}_k\gtrsim N_k\left(\frac{a_k}{\h(2^{-k})}\right)^2.
\end{equation} 
Combining the two bounds,
\begin{eqnarray*}
\#J_{k} & \gtrsim &
(\#\mathcal{T}_k)^{1/2}\h^{-1}(ra_k)^{1/2}\\
 & \gtrsim & N_k^{1/2}\frac{a_k}{\h(2^{-k})}\h^{-1}(ra_k)^{1/2}.
\end{eqnarray*}
Therefore, for any $h\in\H$ as in the hypothesis of the theorem, we have the estimate 
\begin{eqnarray}\label{eq:dompower}
\sum_j h(r_j)& \gtrsim &\sum_{k} \frac{(\h\sqrt{\g})(2^{-k})}{\Delta(2^{-k})}\#J_k\\
& \gtrsim &\sum_{k} \frac{a_k\h^{-1}(ra_k)^{\frac{1}{2}}\sqrt{\g}(2^{-k})N_k^{\frac{1}{2}}}{\Delta(2^{-k})}.
\end{eqnarray}
Recall now  that from \prettyref{eq:entropy_Hinfty} we have $\sqrt{\g}(2^{-k})N_k^{\frac{1}{2}}\gtrsim \HH{\g}_\infty(L_k)^\frac{1}{2}$. In addition, $\h(x)\lesssim x^\alpha$, which implies that $\h^{-1}(x)\gtrsim x^{\frac{1}{\alpha}}$. Therefore we obtain the bound
\begin{eqnarray*}
\sum_j h(r_j)& \gtrsim &  \sum_{k}\frac{\HH{\g}_\infty(L_k)^{1/2}a_k^{\frac{1+2\alpha}{2\alpha}}}{\Delta(2^{-k})}\\
&=& \sum_{k}\HH{\g}_\infty(L_k)^{1/2}\gtrsim 1\nonumber.
\end{eqnarray*}
In the last inequality, we used the $\sigma$-subadditivity of $\HH{\g}_\infty$.
\end{proof}

\begin{corollary}\label{cor:corohsqrtg}
	Let $E$ be an $F^+_{\alpha\beta}$-set for $\alpha,\beta>0$. If $h$ is a dimension function satisfying $h(x)\ge C x^{\frac{\beta}{2}+\alpha} \log^{\theta}(\frac{1}{x})$ for $\theta>\frac{1+2\alpha}{2\alpha}$, then $\HH{h}(E)>0$.

\end{corollary}

\begin{remark}
Note that again the bound $\dim(E)\ge \alpha+\frac{\beta}{2}$ for $E\in F_{\alpha\beta}$ follows directly from this last corollary.
\end{remark}

In the next theorem we consider the case of a family of very small Furstenberg sets. More precisely, we deal with a family that corresponds to the case $\alpha=0$, $\beta\in (0,1]$ in the classical setting.

\begin{theorem}[$\h\g\to\h\sqrt{\g}$, $\h$ zero dimensional, $\g$ positive]\label{thm:hgtohsqrtg,hzero,gpos}
	Let $\beta>0$ and define $\g(x)=x^\beta, \h(x)=\frac{1}{\log(\frac{1}{x})}$.  If $E$ is an $F_{\h\g}$-set, then $\dim(E)\ge \frac{\beta}{2}$.
\end{theorem}

\begin{proof}
Once again, we follow \cite{mr10a}, Theorem $5.1$. Let $E\in F_{\h\g}$ and let $\{B_j\}_{j\in\N}$ be a covering of $E$ by balls with $B_j=B(x_j;r_j)$. Now we consider a scaling sequence $\b$ to be determined later and, by using \prettyref{lem:part}, we obtain a decomposition $L=\bigcup_{k\ge k_0}L_k$ with
\begin{equation*}
	L_k=\left\{e\in L: \hh{\h}(\ell_e\cap E_k)\ge ck^{-2}\right\},
\end{equation*}
where $E_k=E\cap \bigcup_{J_k^\b}B_j$, $J_k^\b$ is the partition of the radii as in \prettyref{eq:jkb} associated to $\b$ and $c>0$ is a suitable constant.  
We apply \prettyref{lem:split} and also define, as in \prettyref{thm:hgtohsqrtg}, $\Pi_k:=J^\b_{k}\times
J^\b_{k}\times\{1,..,N_k\}$ and 
\begin{equation}
\mathcal{T}^\b_k:=\left\{(j_{-},j_{+},i)\in \Pi_k:
    	I_{e_{i}}^-\cap E_{k}\cap B_{j_-}\neq\emptyset\ 
	I_{e_{i}}^+\cap E_{k}\cap B_{j_+}\neq\emptyset
 \right\},\nonumber
\end{equation}
where $\{e^k_j\}_{j=1}^{N_k}$ is a $b_k$-separated subset of $L_k$. 
By \prettyref{lem:conobolas}, we obtain
\begin{equation}\label{eq:otraescala}
\#\mathcal{T}^\b_k\lesssim\frac{b_{k-1}}{b_k}\frac{1}{\h^{-1}(ck^{-2})}(\#J^\b_{k})^{2}.
\end{equation} 
For the lower bound on $\#\mathcal{T}^\b_k$, we have the extra information about the entropy of $L_k$, i.e., $N_k\gtrsim \HH{\beta}_\infty(L_k)/b_k^\beta$. We therefore obtain the analogous of \prettyref{eq:ifijo}:
\begin{equation*}
\#\mathcal{T}_k\gtrsim \frac{\HH{\beta}_\infty(L_k)}{b_k^\beta}\left(\frac{k^{-2}}{\h(b_{k-1})}\right)^2. 
\end{equation*} 
The last two inequalities together imply that 
\begin{equation}
\#J^\b_{k} \gtrsim \HH{\beta}_\infty(L_k)^{\frac{1}{2}}\left(\frac{b_k^{1-\beta}}{b_{k-1}}\right)^{1/2}\frac{e^{-{c}k^2}}{k^2}.\nonumber
\end{equation}
It follows then that, for $s<\frac{\beta}{2}$,
\begin{eqnarray*}\label{eq:sumafinal}
\sum_j r_j^{s}& \ge &\sum_k b_k^{s}\#J_k\\
& = & \sum_k \HH{\beta}_\infty(L_k)^\frac{1}{2} \frac{b_k^{\frac{1}{2}+s-\frac{\beta}{2}}}{b^\frac{1}{2}_{k-1}} \frac{1}{k^2e^{ck^2}}\\
	&\gtrsim & \sqrt{\sum_k \HH{\beta}_\infty(L_k) \frac{b_k^{1+2s-\beta}}{b_{k-1}} \frac{1}{k^4e^{ck^2}}}.
\end{eqnarray*}
Consider the hyperdyadic scale $b_k=2^{-(1+\e)^k}$ with some $\e>0$ to be determined. With this choice, we have
\begin{equation}
	\frac{b_k^{1+2s-\beta}}{b_{k-1}}=2^{(1+\e)^{k-1}-(1+\e)^{k}(1+2s-\beta)}=2^{(1+\e)^{k}\left(\frac{1}{1+\e}-(1+2s-\beta)\right)}.\nonumber
\end{equation} 
Since $1+2s-\beta<1$, we can choose $\e>0$ such that 
$\frac{1}{1+\e}-(1+2s-\beta)>0$.
More precisely, take $\e$ such that 
$0<\e<\frac{\beta-2s}{1+2s-\beta}.$

Therefore, 
\begin{eqnarray*}
\left(\sum_j r_j^{s}\right)^2 & \gtrsim & 
\sum_k \HH{\beta}_\infty(L_k) \frac{b_k^{1+2s-\beta}}{b_{k-1}}\frac{1}{k^4e^{ck^2}}\\
& = &\sum_k \HH{\beta}_\infty(L_k) \frac{2^{(1+\e)^{k}\left(\frac{1}{1+\e}-(1+2s-\beta)\right)}}{k^4e^{ck^2}}\\
&\gtrsim& \sum_k \HH{\beta}_\infty(L_k) \gtrsim 1.
\end{eqnarray*}

\end{proof}

We have the following immediate corollary.

\begin{corollary}
	Let $\theta>0$. If $E$ is an  $F_{\h\g}$-set with $\h(x)=\frac{1}{\log^\theta(\frac{1}{x})}$ and $\g(x)=x^\beta$, then $\dim(E)\ge \frac{\beta}{2}$.
\end{corollary}

The next question would be: Which should it be the expected dimension function for an $F_{\h\g}$-set if 
$\h(x)=\g(x)=\frac{1}{\log(\frac{1}{x})}$? The preceding results lead us to the following conjecture
\begin{conjecture}
Let  $\h(x)=\g(x)=\frac{1}{\log(\frac{1}{x})}$ and let $E$ be an $F_{\h\g}$-set. Then  $\frac{1}{\log^\frac{3}{2}(\frac{1}{x})}$ should be an appropriate dimension function for $E$, in the sense that a logarithmic gap can be estimated.
\end{conjecture}
We do not know, however,  how to prove this.

\section{A remark on the notion of size for the set of directions}\label{sec:size}

In \prettyref{sec:prelim} we have emphasized that the relevant ingredient for the combinatorial proof in \prettyref{sec:hsqrtg} is the notion of $\delta$-entropy of a set. In addition, we have discussed the possibility of consider the Box dimension as an adequate notion of size to detect this quantity. In this section we present an example that shows that the notion of Packing dimension is also inappropriate. We want to remark here that none of them will give any further (useful) information to this problem and therefore there is no chance to obtain similar results in terms of those notions of dimensions. To make it clear, consider the classical problem of proving the bound $\dim(E)\ge\alpha+\frac{\beta}{2}$ for any $E\in F_{\alpha\beta}$ where $\beta$ is the Box or Packing dimension of the set $L$ of directions. 

We illustrate this remark with the extreme case of $\beta=1$. It is absolutely trivial that nothing meaningful can be said if we only know that the Box dimension of $L$ is $1$, since any countable dense subset $L$ of $\s$ satisfies $\dim_B(L)=1$ but in that case, since $L$ is countable, we can only obtain that $\dim(E)\ge\alpha$. 

For the Packing dimension, it is also easy to see that knowing that $\dim_P(L)=1$ does not add any further information about the Hausdorff dimension of the set $E$. To see why, consider the following example. Let $C_\alpha$ be a regular Cantor set such that $\dim(C_\alpha)=\dim_B(C_\alpha)=\alpha$. Let $L$ be a set of directions with $\dim_H(L)=0$ and $\dim_P(L)=1$. We build the Furstenberg set $E$ in polar coordinates as
\begin{equation}
 E:=\{(r,\theta): r\in C_\alpha, \theta\in L\}.
\end{equation} 
This can be seen as a ``Cantor target'', but with a fractal set of directions instead of the whole circle. By the Hausdorff dimension estimates, we know that $\dim(E)\ge\alpha$. We show that in this case we also have that $\dim(E)\le \alpha$, which implies that in the general case this is the best that one could expect, even with the additional information about the Packing dimension of $L$. For the upper bound, consider the function $f:\RR\to\RR$ defined by $f(x,y)=(x\cos y,x\sin y)$. Clearly $E=f(C_\alpha\times L)$, and therefore
\begin{equation*}
 \dim(E)=\dim(f(C_\alpha\times L)) \le \dim (C_\alpha\times L)= \dim_B(C_\alpha)+\dim(L)=\alpha
\end{equation*} 
by the known product formulae that can be found, for example, in \cite{fal03}.

\section{Acknowledgements}
We would like to thank to Pablo Shmerkin for many interesting suggestions an remarks about this problem, specifically the example of the ``Cantor target'' in the last section.

\end{document}